\RequirePackage{fix-cm}
\documentclass{article}
\usepackage{fullpage}

\usepackage[utf8]{inputenc}
\usepackage{amsthm,amssymb,amsmath,bbm}
\usepackage{caption,subcaption}
\usepackage{epsfig,graphicx,graphics,color,tikz,tikz-cd,tcolorbox}
\usepackage{enumerate,enumitem}
\usepackage[sort]{cite}
\usepackage{hyperref}
\usepackage{newtxmath}
\usepackage[normalem]{ulem}
\hypersetup{
    colorlinks=true,       % false: boxed links; true: colored links
    linkcolor=blue,        % color of internal links (change box color with linkbordercolor)
    citecolor=magenta,         % color of links to bibliography
    filecolor=magenta,     % color of file links
    urlcolor=cyan,         % color of external links
    linktoc=all
}

\def\final{0}  % set this to 1 to get a comment-free version
\def\iflong{\iffalse}
\ifnum\final=0  %namely if we allow comments in the output
\newcommand{\kristof}[1]{{\color{red}[{\textbf{Kristóf:}  #1}]\marginpar{\color{red}*}}}
\newcommand{\marci}[1]{{\color{blue}[{ \textbf{Marci:}  #1}]\marginpar{\color{blue}*}}}

\newcommand{\tlaci}[1]{{\color{purple}[{ \textbf{TLaci:}  #1}]\marginpar{\color{purple}*}}}
\else % in this case [final=1] we don't want any comments to show
\newcommand{\kristof}[1]{}
\newcommand{\marci}[1]{}
\newcommand{\tlaci}[1]{}
\fi

\makeatletter
\newcommand{\mylabel}[2]{#2\def\@currentlabel{#2}\label{#1}}
\makeatother

\theoremstyle{plain}
\newtheorem{thm}{Theorem}[section]
\newtheorem{lem}[thm]{Lemma}
\newtheorem{cor}[thm]{Corollary}
\newtheorem{cl}[thm]{Claim}
\newtheorem{prop}[thm]{Proposition}

\newtheorem{qu}[thm]{Question}

\theoremstyle{definition}

\newtheorem{ex}[thm]{Example}

\newtheorem{rem}[thm]{Remark}

\newcommand*{\claimproofname}{Proof of claim.}
\newenvironment{claimproof}[1][\claimproofname]{\begin{proof}[#1]}{\end{proof}}

\newcommand{\R}{\mathbb{R}}
\newcommand{\Z}{\mathbb{Z}}
\newcommand{\N}{\mathbb{N}}

\newcommand{\dH}[1]{d_{\Haus_{2^{#1}}}} % Erre a távolságra valami jobb jelölés jó lenne.

\newcommand{\bR}{\mathbb{R}}
\newcommand{\bZ}{\mathbb{Z}}

\newcommand{\cB}{\mathcal{B}}

\newcommand{\cH}{\mathcal{H}}

\newcommand{\cI}{\mathcal{I}}

\newcommand{\cP}{\mathcal{P}}
\newcommand{\cR}{\mathcal{R}}
\newcommand{\Tb}{\mathcal{T}}

\newcommand{\rhog}{\rho_\Gb}

\def\Gb{\mathbf{G}}\def\Hb{\mathbf{H}}

\def\Tb{\mathbf{T}}

\newcommand{\mm}{\text{\sf mm}}
\newcommand{\mmp}{\mm_+}
\newcommand{\bmm}{\text{\sf bmm}}
\newcommand{\bmmp}{\bmm_+}

\newcommand*\diff{\mathop{}\!\mathrm{d}}

\newcommand{\fg}{\varphi}
\newcommand{\eps}{\varepsilon}
\newcommand{\Haus}{{\text{\rm Haus}}}
\def\one{{\mathbbm1}}

\newcommand{\cost}{\textrm{\rm cost}}
\newcommand{\conv}{\textrm{conv}}

\usepackage{titling}

\title{Cycle Matroids of Graphings: From Convergence to Duality}

\thanksmarkseries{alph}
\author{
Kristóf Bérczi\thanks{MTA-ELTE Matroid Optimization Research Group and HUN-REN–ELTE Egerváry Research Group, Department of Operations Research, Eötvös Loránd University, and HUN-REN Alfréd Rényi Institute of Mathematics, Budapest, Hungary. Email: \texttt{kristof.berczi@ttk.elte.hu}.}
\and
Márton Borbényi\thanks{Department of Computer Science, Eötvös Loránd University and HUN-REN Alfréd Rényi Institute of Mathematics, Budapest, Hungary. Email: \texttt{marton.borbenyi@ttk.elte.hu}.}
\and
László Lovász\thanks{HUN-REN Alfréd Rényi Institute of Mathematics, Budapest, Hungary. Email: \texttt{laszlo.lovasz@ttk.elte.hu}.}
\and
László Márton Tóth\thanks{HUN-REN Alfréd Rényi Institute of Mathematics, Budapest, Hungary. Email: \texttt{toth.laszlo.marton@renyi.hu}.}
\\[-3.0ex] 
}

\date{}

\begin{document}

\maketitle
\tableofcontents

%%%%%%%%%%%%%%%%%%%%%%%%%%%%%%%%
 \newpage
% \pagenumbering{roman}
% \tableofcontents
% \newpage
% \pagenumbering{arabic}
% \setcounter{page}{1}
%%%%%%%%%%%%%%%%%%%%%%%%%%%%%%%%

\begin{abstract}
A recent line of research has concentrated on exploring the links between analytic and combinatorial theories of submodularity, uncovering several key connections between them. In this context, Lovász initiated the study of matroids from an analytic point of view and introduced the cycle matroid of a graphing~\cite{lovasz2023matroid}. Motivated by the limit theory of graphs, the authors introduced a form of right-convergence, called quotient-convergence, for a sequence of submodular setfunctions, leading to a notion of convergence for matroids through their rank functions~\cite{berczi2024quotient}. 

In this paper, we study the connection between local-global convergence of graphs and quotient-convergence of their cycle matroids. We characterize the exposed points of associated convex sets, forming an analytic counterpart of matroid independence- and base-polytopes. Finally, we consider dual planar graphings and show that the cycle matroid of one is the cocycle matroid of its dual if and only if the underlying graphings are hyperfinite.
\medskip

\end{abstract}

\section{Introduction}

Cycle matroids bridge graph theory and matroid theory by providing matroid theoretic counterparts of graph concepts such as cycles, cuts, and spanning trees, thus leading to a powerful framework for understanding and solving various combinatorial and optimization problems. The matroidal perspective allows for a deeper analysis of graph properties and leads to efficient algorithms in network design, electrical engineering, and operations research. In particular, the cycle matroid of a graph $G$ is characterized by its rank function, where the rank of a set $F\subseteq E(G)$ of edges is $r_G(F)=|V(G)|-|\{\textrm{$F$-connected components of $G$}\}|$; see Section~\ref{sec:prelim} for properties of matroid rank functions.

There has been an extensive line of research on extending the notion of cycle matroid to infinite graphs, for example for connected locally finite graphs. Probably the simplest notion for infinite graphs is the finitary cycle matroid whose circuits are exactly the finite cycles. However, an unfortunate feature of finitary matroids is that excluding infinite circuits spoils duality. As a workaround, Bruhn, Diestel, Kriesell, Pendavingh and Wollan~\cite{bruhn2013axioms} proposed five equivalent sets of axioms leading to a non-finitary theory of infinite matroids. These axioms were in fact shown to be equivalent to the notion of B-matroids introduced by Higgs~\cite{higgs1969matroids}. Using these axioms, Bruhn and Diestel~\cite{bruhn2011infinite} provided further examples by defining the circuits of the matroid to be the edge sets of topological circuits of compactifications of infinite graphs obtained by adding the space of ends to the graph. Later, Bowler, Carmesin and Christian~\cite{bowler2018infinite} introduced a class of infinite cycle matroids that unifies previous constructions. Motivated by the work of Ne\v{s}et\v{r}il and Ossona de Mendez~\cite{nesetril} on first order convergence of graphs, Kardoš, Král', Liebenau, and Mach~\cite{kardovs2017first} introduced the notion of first order convergence of linearly representable matroids. Since graphic matroids are linearly representable, this also provides a convergence notion for such matroids. 

In addition to the results discussed above, however, there have been no or only moderate work concerning graphs with an underlying measurable structure. In this paper we study the cycle matroids of graphings, as graphings are natural limit objects in the context of sparse graph limits. A \emph{graphing} is a bounded degree graph $\Gb=(J,\mu,E)$, where $(J,\mu)$ is a standard Borel probability space, the edge set $E$ is a Borel subset of $\binom{J}{2}$ with a certain measure preserving property; this property allows one to define the measure $\eta_\Gb$ on the edge set, see Subsection~\ref{subsec:prelim_graphing} for the full definition. Graphings serve as limit objects in sparse graph limit theory. Recently, Lovász~\cite{lovasz2023matroid} initiated a limit theory of cycle matroids of bounded-degree graphs, and introduced the cycle matroid of a graphing $\Gb =(J,\mu,E)$ by defining its rank function as
\begin{equation} \label{eqn:graphing_rank}
    \rho_\Gb(F) = 1- \mathbb{E}_x\left[ \frac{1}{|V(\Gb^F_x)|} \right],
\end{equation}
where $F \subseteq E$ is Borel, $x$ is a uniform random point from $\mu$, $\Gb^F_x$ denotes its connected component in the graphing $\Gb^Y=(J,\mu, F)$, and $V(\Gb^F_x)$ denotes the set of vertices of the component. In~\cite{lovasz2023matroid}, it was shown that $\rho_\Gb$ is an increasing submodular setfunction on the Borel subsets of $E$, and $\rho_\Gb\le\eta_\Gb$. It is not difficult to verify that this definition generalizes the rank function of the cycle matroid of a finite graph, normalized by the number of vertices. In particular, if $\Gb$ a forest with finite components, then $\rho_\Gb(E)$ is half of the average degree. If $\Gb$ is any graphing with infinite components, then $\rho_\Gb(E)=1$. 

This raises several questions about cycle matroids of graphings. How to define and describe their independent sets, bases, and other basic properties known from the finite case? Does the convergence of a sequence of finite graphs to a graphing imply convergence of the corresponding matroids (in the sense introduced in~\cite{berczi2024quotient})? Does duality of matroids play any role?  In this paper, we provide some answers to these questions. We develop the limit theory of cycle matroids of bounded-degree graphs, show that cycle matroids of graphings are their limit objects in the sense of~\cite{berczi2024quotient}, we study their bases, and investigate how they behave under duality of planar graphings. 

\paragraph{Limit theory.}
Our first result connects local-global convergence of graphs with quotient-convergence of the rank function of cycle matroids, defined in~\cite{berczi2024quotient}. A \emph{$k$-quotient} of a setfunction $\varphi$ on a set-algebra $(J, \mathcal{B})$ is $\varphi \circ F^{-1}$, where $F: J \to [k]$ is a measurable map. We denote by $Q_k(\varphi)$ the set of $k$-quotients of $\varphi$. A sequence of setfunctions $(\varphi_1, \varphi_2, \dots)$ \emph{quotient converges} to $\varphi$, denoted $\varphi_n \rightarrowtail \varphi$, if $Q_k(\varphi_n) \to Q_k(\varphi)$ as $n\to\infty$ in the Hausdorff distance for all $k \in \mathbb{N}$. Note that $Q_k(\varphi_n)$ and $Q_k(\varphi)$ are subsets of the same $2^k$-dimensional linear space, so convergence in Hausdorff distance does not depend on the norm of this space. We refer to Subsection~\ref{subsec:quotient_convergence} for more details on this convergence notion. Our first main theorem concerns the case of cycle matroids of bounded-degree graphs.

\begin{thm} \label{thm:intro_global-conv}
Let $(G_1,G_2,\dots)$ be a sequence of finite graphs with all degrees bounded by $D$ that converges to the graphing $\Gb$ in the local-global sense. Let $r_{G_n}$ denote the rank function of the cycle matroid of $G_n$. Then
\[
\frac{r_{G_n}}{|V(G_n)|} \rightarrowtail \rho_{\Gb}
\]
in the sense of quotient-convergence of submodular setfunctions. 
\end{thm}

We also exhibit an example showing that assuming local convergence of $G_n$ is not enough: there exists a sequence $(H_1,H_2,\dots)$ of finite graphs with all degrees bounded by $D$ that converges to a graphing $\Gb$ in the local sense, however, $\frac{r_{H_n}}{|V(G_n)|}\centernot\rightarrowtail\rho_{\Gb}$; see Example~\ref{ex:loc} for details.

\paragraph{Independent sets.}
Our goal is to define independent sets of the cycle matroid of a graphing. There are various ways to characterize independent sets in the cycle matroid of a finite graph, of which we consider the following four. 
%(These notions will lead to the something similar in the case of graphings.)
\begin{itemize}\itemsep0em
    \item Independent sets in the cycle matroid of a finite graph are edge-sets of acyclic subgraphs (briefly, forests).
    \item Bases of the cycle matroid of a finite graph are spanning forests, i.e. they are spanning trees in every connected component. The independent sets are their subsets.
    \item Independent sets in any matroid can be recovered from the rank function as sets with $r(X)=|X|$.
    \item In a finite matroid $(E,r)$, solutions of the system of inequalities $x_e\geq 0$ $(e\in E)$, and $\sum_{e \in A} x(a) \leq r(A)$ $(A \subseteq E)$ are called {\it fractional independent sets}, and form the {\it matroid polytope}. Vertices of the matroid polytope are precisely the (indicator vectors of) independent sets.
\end{itemize}
Before generalizing these definitions for graphings and stating the theorem about a certain equivalence of them, we need to discuss some points where the generalizations are not quite straightforward.

In this section, we consider forests that are hyperfinite. A graphing $\Gb=(J,\mu,E)$ is \emph{hyperfinite} if, for every $\eps > 0$, there is a Borel set $S \subseteq E$ with $\eta_\Gb(S) \leq \eps$ such that every connected component of $E\setminus S$ is finite. A \emph{forest} in $\Gb$ is an acyclic subset $T\subseteq E$ of the edges. An \emph{essential spanning forest} is a forest that forms a spanning tree in every finite component of $\Gb$, and contains only infinite $T$-components in the infinite $\Gb$-components. Note that $T$ might contain several components in an infinite component of $\Gb$, but all of those are required to be infinite; this is why $T$ is called an essential spanning forest instead of a spanning tree. Let us emphasize that we will call a set of edges a ``forest'' or an ``essential spanning forest'' if they became one after adding or deleting a set of edges of measure $0$.

A vertex of a convex polytope can be generalized to the infinite setting in several ways. A point $v$ of the convex set $C$ in a Banach space $X$ is \emph{extreme}, if it cannot be written as $v=\frac12(u+w)$, where $u,w\in C\backslash\{v\}$. The point $v$ is \emph{exposed}, if there is a
continuous linear functional $\Lambda$ on $X$ such that $\Lambda(y)\le\Lambda(x)$ for every $y\in C$, and $x$ is the unique point of $C$ with
equality. An exposed point is necessarily extreme, but not the other way around; see Section~\ref{sec:extreme} for further comments.
 
It was suggested in~\cite{lovasz2023submodular} that for a general increasing submodular setfunction $\fg$ with $\fg(\emptyset)=0$, the generalization of the matroid polytope should be the set of its minorizing charges, i.e.\ finitely additive measures $\alpha$ such that $0\le\alpha\le\fg$. For a graphing $\Gb=(J,\mu,E)$, we are in a simpler situation, since a finitely additive measure $\alpha$ minorizing $\rho_\Gb$ satisfies $\alpha\le\eta_\Gb$, and so it is automatically countably additive (i.e., a proper measure). So the analytic version of the cycle matroid polytope will be 
\[
\mmp(\rho_{\Gb}) = \big\{\alpha \textrm{ measure on $E$}\mid  0\leq \alpha \leq \rho_{\Gb} \big\}.
\] 
The base polytope of a cycle matroid can be generalized as
\[
\bmmp(\rho_\Gb) = \big\{\alpha\in\mmp(\rho_\Gb)\mid  \alpha(E)=\rho_\Gb(E) \big\}.
\]
First, we give a characterization of the exposed points of $\mmp(\rho_\Gb)$ and $\bmmp(\rho_\Gb)$.

\begin{thm}\label{thm:exposed}
Let $\Gb=(J,\mu,E)$ be a graphing. 
\begin{enumerate}[label=(\alph*)]\itemsep0em
    \item $\alpha\in\mmp(\rho_\Gb)$ is exposed if and only if $\alpha=\eta_\Gb|_F$ for a hyperfinite subforest $F\subseteq E$. \label{exposed:a}
    \item $\alpha\in\bmmp(\rho_\Gb)$ is exposed if and only if $\alpha=\eta_\Gb|_F$ for a hyperfinite essential spanning forest $F\subseteq E$. \label{exposed:b}
\end{enumerate}
\end{thm}

We also give characterization of these kind of subsets used in Theorem~\ref{thm:exposed}.

\begin{thm}\label{thm:exposed_char}
Let $\Gb=(J,\mu,E)$ be a graphing and $F\subseteq E$ be a Borel subset of edges. Then the following are equivalent:
    \begin{enumerate}[label=(\roman*)]\itemsep0em
    \item  $(J,\mu,F)$ is a hyperfinite forest; \label{indepi}
    \item $F$ is a Borel subset of the edge set of a hyperfinite essential spanning forest; \label{indepii}
    \item $\rho_\Gb(F)=\eta_\Gb(F)$. \label{indepiii}
   % \item the measure $\alpha$ is an exposed point of $\mmp(\rho_\Gb)$. \label{indepiv}
\end{enumerate}
Similarly, the following are equivalent:
\begin{enumerate}[label=(\roman*')]\itemsep0em
    \item $(J,\mu,F)\subseteq(J,\mu,E)$ is a hyperfinite essential spanning forest; \label{basei}
    %\item $Y$ is a Borel subset of the edge set of a hyperfinite essential spanning forest;
    \item  $\rho_\Gb(F)=\eta_\Gb(F)=\rho_\Gb(E)$. \label{baseii}
    %\item the measure $\alpha$ is an exposed point of $\bmmp(\rho_\Gb)$. \label{baseiii}
\end{enumerate}
\end{thm}

Along the lines, we introduce the notion of {\it weakly exposed} points and prove structural properties of those, see Theorem~\ref{thm:weakly} and Corollary~\ref{cor:weakstrong}. In Section~\ref{sec:submod}, we show that hyperfinite forests satisfy an exchange property, which in turn leads to a conceptually new proof of the submodularity of the set function $\rho_\Gb$, different from the (somewhat lengthy) proof in~\cite{lovasz2023matroid}. The following question remains open.

\begin{qu}
Are all extreme points of $\bmmp(\rho_\Gb)$ exposed, and therefore hyperfinite spanning forests? 
\end{qu}

\paragraph{Duality.} Duality is an important concept of matroid theory, and it lines up nicely with duality of finite planar graphs: the dual of the cycle matroid of a planar map (embedded connected planar graph) is the cycle matroid of its planar dual. The question naturally arises: Does an analogous statement hold for planar graphings? 

The notion of a ``planar graphing'' takes some work to set up. In order to talk about pairs of dual planar graphings, we define planar graphings $\Gb=(J,\mu,E)$ in a rather restricted sense. The connected components of a \emph{graphing map} are 2-vertex-connected graphs endowed with a specified planar embedding. We assume that the vertex sets of these planar maps have no accumulation points in the plane, and all faces are finite cycles, whose lengths are uniformly bounded. Then the dual graphing map $\Gb^*=(J^*,\mu^*,E^*)$ is well defined, and it has bounded degrees. For a dual pair of graphing maps, there is a measure preserving Borel bijection $\sigma\colon E\to E^*$ such that the image of the star of a vertex of $\Gb$ is a facial cycle in $\Gb^*$ and vice versa. We also note that the hyperfiniteness of $\Gb$ and of $\Gb^*$ are equivalent~\cite[Example 3.11, Proposition 3.12]{angel2018hyperbolic}. See Section~\ref{sec:duality} for an exact definition, and also~\cite{angel2018hyperbolic} and~\cite{conley2021one} for a more general treatment.

\begin{rem}
The above assumptions are made solely to avoid technical difficulties, not because they are necessary for the results presented here. Our aim in this paper is to showcase how duality works, and these assumptions allow us to do so with minimal distraction. The aforementioned works~\cite{angel2018hyperbolic} and~\cite{conley2021one} demonstrate that, with dedicated treatment, such technical questions can be addressed to prove theorems for planar graphings without these restrictions.
\end{rem}

The dual of the cycle matroid, called the \emph{cocycle} matroid of $\Gb=(J,\mu,E)$, can be defined for any (not necessarily planar) graphing by its rank function 
\[
\rho^*_{\Gb}(A) = \rho_{\Gb}(E\setminus A) + \eta_\Gb  (A) - \rho_{\Gb} (E).
\]  
For a general notion of duality for submodular setfunctions, see~\cite{berczi2024monotonic}. Our first result here shows that in the hyperfinite case duality works similarly as in the finite case.

\begin{thm} \label{thm:hyperfinite_duality}
    Let $\Gb$ and $\Gb^*$ be a hyperfinite pair of dual planar graphings. Then $\rho^*_{\Gb} = \rho_{\Gb^*}$ and $\rho^*_{\Gb^*} = \rho_{\Gb}$.
\end{thm}

However, when $\Gb$ and $\Gb^*$ are not hyperfinite, the two matroids do not form a dual pair. This is the case, for example, if the components of $\Gb$ and $\Gb^*$ are dual hyperbolic tilings. This example leads us to an interesting area, motivated by group theory, namely the theory of ``cost''. The \emph{cost} of a graphing $\cost(\Gb)$ is a nonnegative number, designed to measure the minimal amount of edges needed to generate the connected components of a graphing; see the precise definition in Subsection~\ref{subsec:prelim_graphing}. Every hyperfinite graphing has cost at most $1$, and exactly $1$ if all components are infinite. On the other hand, if a planar graphing is not hyperfinite, then its cost is strictly more than $1$~\cite{conley2021one}. 

\begin{thm} \label{thm:cost_achieveing_matroid}
    Let $\Gb$ and $\Gb^*$ be a non-hyperfinite pair of dual planar graphings. Then $\rho^*_{\Gb^*} \neq \rho_{\Gb}$. Moreover, $\rho^*_{\Gb^*}$ is a submodular function on $E$ that attains the cost, i.e. $\rho^*_{\Gb^*}(E) = \cost(\Gb)$.  
\end{thm}

Both submodular functions $\rho_{\Gb}$ and $\rho^*_{\Gb^*}$ are natural to consider, as supported by the following. Let $\Gb$ be a graphing and $<$ be an arbitrary Borel total order on the edges. We define the {\it free maximal spanning forest} $\Tb_f$ of $\Gb$ by simultaneously deleting all edges $e$ of $\Gb$ for which there exists a cycle $C$ containing $e$ in which $e$ is $<$-minimal. The {\it wired maximal spanning forest} $\Tb_w$ is similar to the free version, but we also remove any edge that is minimal with respect to the ordering on a two-way infinite path. It is straightforward to check that both $\Tb_f, \Tb_w$ are acyclic and $\Tb_w\subseteq \Tb_f$. We refer to~\cite[Section 11]{lyons2017probability} for a thorough introduction. It is also true that $\Tb_w$ is an essential spanning forest of $\Gb$. We show the following.

\begin{thm}\label{thm:free_vs_wired}
    Let $\Gb$ and $\Gb^*$ be a pair of dual planar graphings and $<$ be an arbitrary Borel total order on the edges. The Wired Maximal Spanning Forest $\Tb^<_w$ is an exposed point of $\bmm^+(\rho_{\Gb})$, while the Free Maximal Spanning Forest $\Tb^<_f$ is an exposed point of $\bmm^+(\rho^*_{\Gb^*})$.
\end{thm}

\paragraph{Matroids beyond the hyperfinite.}
While the definition of the cycle matroid via (\ref{eqn:graphing_rank}) is a consistent generalization of the finite notion, it does have a drawback: the rank of the full edge set is at most $1$. This is natural for finite graphs, as spanning trees have average degree slightly below 2. For graphings, however, spanning trees can have higher average degree, so if there is a notion of a matroid whose bases are spanning trees instead of essential spanning trees, then the total rank should be larger than $1$. Theorems \ref{thm:cost_achieveing_matroid} and \ref{thm:free_vs_wired} above show that for planar graphings, the dual matroid of the dual graphing provides such a submodular rank function. 

To further justify the study of $\rho_{\Gb}$, we highlight that even though $\rho_{\Gb}$ only sees spanning forests that are hyperfinite, it still carries a lot of information about the graphing. In the finite case, Whitney proved that every 3-connected graph is uniquely determined by its cycle matroid. In an upcoming work~\cite{uniqueness}, we study analogous problems for graphings.

\begin{qu} \label{qu:whitney}
Under what assumptions does $\rho_\Gb$ determine the graphing $\Gb$? 
\end{qu}

As a first step in this direction, we could show that $\rho_{\Gb}$ already determines acyclic graphings with minimum degree 3 up to \emph{isomorphism} in the following sense: For a graphing $\Gb$ and an acyclic graphing $\Tb$ of minimum degree 3, if $\rho_{\Gb} \cong \rho_{\Tb}$ then $\Gb \cong \Tb$. In both cases, $\cong$ means that there exists a measure preserving bijection between the ground spaces, up to measure zero, that maps one structure to the other. 

%%%%%%%%%%%%%%%%
\section{Preliminaries} 
\label{sec:prelim}
%%%%%%%%%%%%%%%%

%%%%%%%%%%%%%%%%
\subsection{Submodular setfunctions and matroids}
%%%%%%%%%%%%%%%%

A setfunction $\fg\colon\cB\to\R$ on a set-algebra $(J,\cB)$ is \emph{submodular} if $\fg(X)+\fg(Y)\geq\fg(X\cap Y)+\fg(X\cup Y)$ holds for $X,Y\in\cB$. One of the most basic examples of a submodular function is the rank function of a matroid. A \emph{matroid} $M=(E,r)$ is defined by its finite \emph{ground set} $E$ and its \emph{rank function} $r\colon2^E\to\bZ_+$ that satisfies the so-called \emph{rank axioms}: (R1) $r(\emptyset)=0$, (R2) $X\subseteq Y\Rightarrow r(X)\leq r(Y)$, (R3) $r(X)\leq |X|$, and (R4) $r(X)+r(Y)\geq r(X\cap Y)+r(X\cup Y)$. Note that axiom (R4) requires submodularity of the rank function. A set $X\subseteq E$ is called \emph{independent} if $r(X)=|X|$ and \emph{dependent} otherwise. The \emph{circuits} of the matroid are formed by its (inclusionwise) minimal dependent sets, while maximal independent sets are called \emph{bases}. The rank axioms imply that any two bases have the same cardinality. For further details on matroids, the reader is referred to~\cite{oxley2011matroid}.

%%%%%%%%%%%%%%%%
\subsection{Graphings and hyperfiniteness} 
\label{subsec:prelim_graphing}
%%%%%%%%%%%%%%%%

We give the basic definitions, but point to~\cite[Chapters 18 and 21]{lovasz2012large} for a thorough introduction. Let $(J, \mu)$ be a standard Borel probability space.
A \emph{graphing} is a graph $\Gb$ with vertex set $V(\Gb) = J$ and Borel edge set $E \subseteq \binom{J}{2}$, in which all degrees are finite, and 
\begin{equation*} \label{eqn:graphing}
\int_{A} \deg(B,x) \diff \mu (x) = \int_{B} \deg(A,x) \diff \mu (x)
\end{equation*}
for all measurable sets $A,B \subseteq J$, where $\deg(S,x)$ is the number of edges from $x \in J$ to $S \subseteq J$. This assumption allows one to meaningfully define the \emph{edge measure} $\eta_\Gb $ on Borel subsets $F\subseteq E$ by setting
\begin{equation*}
\eta_\Gb(F)=\frac{1}{2}\int \deg_F(x)\diff \mu (x),
\end{equation*}
where $\deg_F(x)$ is the degree of $x$ in $F$. The measure $\eta_\Gb $ is then concentrated on $E_\Gb$. It is worth noting that $F'\subseteq F\subseteq E$ implies $\eta_{\Gb^F}(F')=\eta_\Gb(F')$, where $\Gb^F=(J,\mu,F)$ is the graphing obtained by deleting the edges of $\Gb$ in $E\setminus F$.

An acyclic graphing, that is, a graphing in which almost all components are trees, is called a \emph{treeing}. We will use the following characterizations of a treeing being hyperfinite.

\begin{lem}\label{lem:levitt}
Let $\Gb=(J,\mu,E)$ be a graphing. Then the following are equivalent:
\begin{enumerate}[label=(\roman*)]\itemsep0em
\item $\Gb$ is a hyperfinite treeing; \label{levitti}
\item almost every component of $\Gb$ is a tree with at most $2$ ends; \label{levittii}
\item $\eta_\Gb (E)=\rho_{\Gb}(E)$. \label{levittiii}
\item $\eta_\Gb=\rho_\Gb$ as setfunctions; \label{levittiv}
\end{enumerate}
\end{lem}
\begin{proof}\mbox{}\\
$\ref{levitti}\Leftrightarrow\ref{levittii}$: The equivalence was proved by Levitt~\cite{levitt1995cost} and Adams \cite{adams1990trees}; a more recent proof can be found in~\cite{aldous-lyons2007}.\\
$\ref{levitti}\Rightarrow\ref{levittiii}$: The implication was proved in \cite{lovasz2023submodular}.\\
%$\ref{levittiii}\Rightarrow\ref{levittiv}$ is obvious.\\
$\ref{levittiii}\Rightarrow\ref{levittiv}$: It was proved in \cite{lovasz2023submodular} that $\rho_\Gb\le \eta_\Gb$. Now, using the subadditivity of $\rho_\Gb$, we get 
\begin{equation*}
\rho_\Gb(F)+\rho_\Gb(E\setminus F)\ge \rho_\Gb(E)=\eta_\Gb(E)=\eta_\Gb(F)+\eta_\Gb(E\setminus F)\ge \rho_\Gb(F)+\rho_\Gb(E\setminus F).
\end{equation*}
Therefore equality holds throughout, implying $\rho_\Gb(B)=\eta_\Gb(B)$ for any Borel set $F\subseteq E$.\\
$\ref{levittiv}\Rightarrow\ref{levitti}$: 
First we prove that $\Gb$ is a treeing, and then we verify hyperfiniteness. 

Assume that $\Gb$ is not a treeing. Then there exists some positive integer $k$ such that the set of vertices that lie on cycles of length $k$ has positive measure. Since each vertex has bounded degree in $\Gb$, every vertex is contained in bounded many $k$-cycles. Hence there exists $C\subseteq E_\Gb$ such that $C$ is a vertex-disjoint union of $k$-cycles, and $\eta_\Gb(C)>0$. Let $C'\subseteq C$ be Borel such that $C'$ contains exactly one edge in each $k$-cycle. Then
\begin{equation*}
\eta_\Gb(E\setminus C')=\rho_\Gb(E\setminus C')=\rho_\Gb(E)=\eta_\Gb(E)>\eta_\Gb(E\setminus C')=\eta_\Gb(E\setminus C'),     
\end{equation*}
a contradiction.

Now we show that $G$ is hyperfinite. Let $U=\{v\in J\mid |V(\Gb_x)|=\infty\}$ and let $F\subseteq E$ denote the set of edges going between vertices in $U$. Then $\eta_\Gb(F)=\mu(U)$. A graphing is hyperfinite if and only if its restriction to the infinite components is hyperfinite; 
%as the finite components are always hyperfinite; 
let us denote the restriction by $\Gb|_U$. Although $\Gb|_U$ is not necessarily a graphing as the underlying measure might not be probability measure, one can scale it back to $1$. Here the measure of the edges is $\eta_\Gb(F)=\rho_\Gb(F)=\int_U 1d\mu(x)=\mu(U)$, thus $\Gb|_U$ is hyperfinite, see \cite[Corollary 8.10]{aldous-lyons2007}.
\end{proof}

The next corollary follows from Lemma~\ref{lem:levitt}.

\begin{cor}\label{cor:levitt}
    Let $\Gb(J,\mu,E)$ be a graphing and $F\subseteq E$ be a Borel subset of edges. Then the following are equivalent:
    \begin{enumerate}[label=(\roman*)]\itemsep0em
        \item $F$ is a hyperfinite essential spanning forest;
        \item $\rho_\Gb(E)=\rho_\Gb(F)=\eta_\Gb(F)$.
    \end{enumerate}
\end{cor}

Lemma~\ref{lem:levitt} is also related to the theory of cost, which we introduce very briefly; see~\cite{gaboriau2023around} and~\cite{kechris2004topics} for more details. For two graphings $\Gb=(J,\mu,E_\Gb)$ and $\Hb=(J,\mu,E_\Hb)$ on the same standard Borel probability space $(J,\mu)$, we write $\Gb \sim \Hb$ if they have the same vertex sets of connected components. The \emph{cost} of $\Gb$ is defined as
 \[
\cost(\Gb) =  \inf \big\{ \eta _{\Hb} ( E_\Hb ) \ \big| \  \Hb \sim \Gb\big\}.
\]

For a graphing $\Gb=(J,\mu,E)$ with all components of size $n$ the cost is $1-\frac{1}{n}$, and for hyperfinite graphings with infinite components the cost is 1. In general, for any Borel subset $F\subseteq E$, we have $\rho_\Gb(F)\leq\cost(\Gb^F)\leq\eta_\Gb(F)$. Levitt~\cite{levitt1995cost} characterized those cases when the first inequality holds with equality.

\begin{thm}[Levitt]\label{thm:levitt}
    Let $\Gb=(J,\mu,E)$ be a graphing and $F\subseteq E$ be a Borel subset of edges. Then $\cost(\Gb^F)$ is attained and $\rho_\Gb(F)=\cost (\Gb^F)$  if and only if $\Gb^F$ is hyperfinite.
\end{thm}

It follows that $\rho_\Gb(F)$ and $\cost(\Gb^F)$, as functions in $F$, might be different. This also follows from the simple observation that while the rank is always submodular, the cost is not necessarily so: for an example, consider the Bernoulli graphing on $T_3\times \mathbb{Z}\times \mathbb{Z}$. Gaboriau~\cite{gaboriau2000cout} proved that one cannot ``cut costs'' on treeings, that is, their edge measure is the same as their cost. This property in fact characterizes treeings.

\begin{thm}[Gaboriau]\label{thm:gaboriau}
    Let $\Gb=(J,\mu,E)$ be a graphing and $F\subseteq E$ be a Borel subset of edges. Then $\Gb^F$ is a treeing if and only if $\cost(\Gb^F) = \eta_{\Gb}(F)$.
\end{thm}

For a graph $G$, we refer to its vertex and edge sets as $V(G)$ and $E(G)$, respectively, and the same notation is used for graphings.

%%%%%%%%%%%%%%%%
\subsection{Local and local-global convergence}
%%%%%%%%%%%%%%%%

Intuitively, a sequence $(G_1,G_2,\dots)$ of finite graphs is said to be \emph{locally} (or \emph{Benjamini-Schramm}) \emph{convergent} if the local behaviour around a random vertex of $G_n$ converges as $n \to \infty$. Formally, let $r \in \N$ denote our radius of sight, and $\alpha$ denote a rooted, connected graph of radius $r$. The \emph{local statistics} $P_r(G,\alpha)$ of a finite graph $G$ are the probabilities $\mathbb{P}_v\big[ B_{G}(v,r) \cong \alpha\big]$,  where the vertex $v \in V(G)$ is chosen uniformly at random, $B_{G}(v,r)$ denotes the $r$-ball around $v$ in $G$, and the isomorphism has to respect the root. One can define $P_r(G,\alpha)$ analogously for a graphing $\Gb$, with the random vertex being chosen according to $\mu$. A sequence $(G_1,G_2,\dots)$ of finite graphs converges locally to a graphing $\Gb$ if the local statistics converge, i.e.\ $P_r(G_n,\alpha) \to P_r(\Gb,\alpha)$ for all $r$ and $\alpha$. 

Local-global convergence is a more refined notion of convergence, see~\cite{hatami2014limits} for an introduction. We use a colored version of the neighborhood statistics. For a fixed vertex-coloring $c\colon V(G) \to [k]$, the \emph{colored neighborhood statistics} $P_{r}(G,c,\alpha)$ is obtained as before: we choose a vertex $v \in V(G)$ uniformly at random, and then consider its colored $r$-neighborhood $\big( B_G(v,r), c \vert_{B_G(r,v)}\big)$. Then the probabilities $P_{r}[G,c]=\big(P_{r}(G,c,\alpha)\big)_{\alpha}$ form a distribution on the set of rooted, colored, connected graphs $\alpha$ of radius at most $r$. Let $Q_{k,r}(G)$ denote the (finite) set of possible colored neighborhood statistics we can obtain by different colorings:

\[Q_{k,r}(G) = \big\{P_{r}[G,c] ~\big|~ c\colon V(G) \to [k]  \big\}.\]

The set $Q_{k,r}(\Gb)$ is defined analogously for a graphing, with the assumption that the coloring $c\colon V(\Gb) \to [k]$ is required to be Borel. The sequence $G_n$ locally-globally converges to $\Gb$, if $Q_{k,r}(G_n) \to Q_{k,r}(\Gb)$ in the Hausdorff distance. The local-global convergence of $G_n$ means that for any $r,k \in \mathbb{N}$ and $\varepsilon >0$, there is a threshold such that, for $i$ large enough, the colored neighborhood distribution $P_{r}[G_i,c]$ for any $k$-coloring $c$ can be approximately modeled on $\Gb$. That is, we can find some Borel coloring $c'$ of $\Gb$ such that $P_{r}[G_i,c]$ and $P_{r}[\Gb,c']$ are $\varepsilon$-close. Vice versa, Borel colorings of $\Gb$ can be approximately modeled on $G_i$.

%%%%%%%%%%%%%%%%
\subsection{Quotient-convergence of setfunctions} \label{subsec:quotient_convergence}
%%%%%%%%%%%%%%%%

The quotient-convergence of setfunctions was introduced in~\cite{berczi2024quotient}. Let $\fg\colon\cB\to\R$ be a setfunction on a set-algebra $(J,\cB)$ with $\fg(\emptyset)=0$. The {\it $k$-quotient set} $Q_k(\fg)$ is then defined as the set of all quotients of $\fg$ on $[k]$, that is, the set of all setfunctions $\fg\circ F^{-1}$ where $F\colon J\to[k]$ is a measurable map. Note that $Q_k(\fg)$ is a subset of $\R^{2^k}$, with the two coordinates $\fg\circ F^{-1}(\emptyset)=0$ and $\fg\circ F^{-1}(J)=\fg(J)$ being fixed. The \emph{quotient profile} of $\fg$ is defined as the sequence $(Q_1(\fg),Q_2(\fg),\dots)$, while the sequence of closures $(\overline{Q}_1(\fg),\overline{Q}_2(\fg),\dots)$ is the {\it closed quotient profile} of $\fg$. A sequence $(\varphi_1,\varphi_2,\dots)$ of setfunctions on set-algebras $(J_1,\cB_1),(J_2,\cB_2),\dots$ is {\it quotient-convergent} if for every positive integer $k$, the quotient sets $(Q_k(\fg_1),Q_k(\fg_2),\dots)$ form a Cauchy sequence in the Hausdorff distance $\dH{k}$ on the $2^k$-dimensional Euclidean space. The sequence $(\fg_1,\fg_2,\dots)$ {\it quotient-converges to $\fg$} for some setfunction $\fg$ if the sequence $(Q_k(\fg_1),Q_k(\fg_2),\dots)$ converges to $Q_k(\fg)$ in the Hausdorff distance, denoted by $\fg_n \rightarrowtail \fg$. For a finite \emph{measurable partition} $\cP = \{P_1,\dots, P_q\}$ of $J$, that is, each $P_i$ is measurable, then $\varphi/\cP\colon 2^{[k]} \to \bR$ denotes the map defined by $(\varphi/\cP)(X) = \varphi(\cup_{i \in X} P_i)$ for $X \subseteq [k]$.

%%%%%%%%%%%%%%%%
\section{Connection to sparse graph limits}
\label{sec:limits}
%%%%%%%%%%%%%%%%

We first prove Theorem~\ref{thm:intro_global-conv} in Section~\ref{sec:continuity}, showing that if a sequence of finite bounded-degree graphs converges to the graphing in the local-global sense, then the corresponding sequence of rank functions quotient-converges to the rank function of the graphing. Then, in Section~\ref{sec:noncon}, we present an example showing that assuming local convergence of the graph sequence is not enough.

%%%%%%%%%%%%%%%%
\subsection{Continuity for local-global convergence}
\label{sec:continuity}
%%%%%%%%%%%%%%%%

The following result appeared in~\cite{lovasz2023matroid}.

\begin{prop}\label{thm:lovasz}
    Let $(G_1,G_2,\dots)$ be a sequence of finite graphs with all degrees bounded by $D$ that converges to a graphing $\Gb$ in the local sense. Then, $r_{G_n}(E(G_n))/|V(G_n)|\to\rho_{\Gb}(E(\Gb))$.
\end{prop}

We can strengthen this result to imply Theorem~\ref{thm:intro_global-conv} by the following observation. Given any graphing parameter $f$, we can turn it into a setfunction $\varphi_f$ on subsets of $E(\Gb)$ by setting
\[
\varphi_{\Gb,f}(F)  = f(\Gb^F).
\]

\begin{thm} \label{thm:local_to_local-global_boost}
Let $(G_1,G_2,\dots)$ be a sequence of finite graphs with all degrees bounded by $D$ that converges to a graphing $\Gb$ in the local-global sense. If $f$ is continuous with respect to local convergence then $\varphi_{\Gb_n,f} \rightarrowtail \varphi_{\Gb,f}$. 
\end{thm}
\begin{proof}
For sake of brevity, let $\varphi_n = \varphi_{\Gb_n,f}$, $\varphi = \varphi_{\Gb,f}$, $Q^n_k=Q_k(\fg_n)$ and $Q_k=Q_k(\varphi)$. We need to prove that for all $k$, $Q^n_k\to Q_k$ in the Hausdorff distance. In other words, we need to show the following:
    \begin{enumerate}[label=(\arabic*)]\itemsep0em
        \item \label{itm:lim_to_seq} if $a\in Q_k$, then there exists $a_n\in Q^n_k$ such that $a_n\to a$,
        \item \label{itm:seq_to_lim} if $a_n\in Q^n_k$ and $a_n\to a$, then $a \in \overline{Q}_k$.
    \end{enumerate}
We prove these using the following technical claim. 

\begin{cl} \label{cl:coloured_conv}
    Let $(\Gb_1,\cP_1), (\Gb_2,\cP_2), \ldots$ be a sequence of edge-coloured graphings with at most $k$ colours that converges as coloured graphings to the coloured graphing $(\Gb,\cP)$ in the local sense. Then, $\varphi_n/\cP_n$ converges to $\varphi/\cP$ in $\R^{2^k}$. 
\end{cl}
\begin{claimproof}
    For $A\subseteq [k]$, let $\tilde{A}_n$ (respectively $\tilde{A}$) denote the set of edges of $\Gb_n$ (respectively $\Gb$) with colour in $A$. Clearly, $\varphi_n(\tilde{A}_n) = (\varphi_n / \cP_n )(A)$, and similarly for $\varphi$, $\cP$, and $\tilde{A}$. We denote by $\cH_{A,n}$ the subgraphing of $\Gb_n$ with edge set $\tilde{A}_n$, and define $\cH_A$ similarly. The coloured local convergence of the sequence  $(\Gb_n,\cP_n)$ implies that $\cH_{A,n}$ locally converges to $\cH_A$ for all $A$. As $f$ is continuous with respect to local convergence, we have 
    \[(\varphi_n/\cP_n)(A) = \varphi_n(\tilde{A}_n) = f(\cH_{A,n}) \to f(\cH_A) = \varphi(\tilde{A})= (\varphi / \cP) (A),\]
    concluding the proof of the claim.
\end{claimproof}

Proof of \ref{itm:lim_to_seq}. As $a\in Q_k$, there is a Borel partition $\cP$ of the edges of $\Gb$ such that $\varphi/\cP=a$. By $\cP$ being a Borel partition and by local-global convergence, we can find some Borel partition $\cP_n$ of the edges of $\Gb_n$ such that the sequence $(\Gb_n, \cP_n)$ locally converges to $(\Gb, \cP)$ as coloured graphings\footnote{Usually one defines local-global convergence using vertex colorings, but edge colorings can also be approximated, see~\cite[Proposition 19.15]{lovasz2012large}.}. By Claim~\ref{cl:coloured_conv}, the corresponding $a_n = \varphi_n / \cP_n$ converges to $a$. 

Proof of \ref{itm:seq_to_lim}. Assume $a_n \to a$. For all $a_n$, there exists a partition $\cP_n$ such that $a_n=\varphi_n/\cP_n$. By possibly taking a subsequence, we may assume that the sequence $(\Gb_n, \cP_n)$ converges locally as coloured graphings. Let $(\Gb^*, \cP^*)$ represent the limit. The graphing $\Gb^*$ a priori has nothing to do with $\Gb$ or the sequence $\Gb_n$ because local-global convergence of the sequence $\Gb_n$ does not imply a partition $\cP$ of the edges of $\Gb$ such that $(\Gb, \cP)$ represents the limit of $(\Gb_n, \cP_n)$. Nevertheless, by completeness, such a coloured graphing $(\Gb^*, \cP^*)$ exists. However, local-global convergence provides arbitrarily good approximations of $(\Gb^*, \cP^*)$ on $\Gb$. That is, there exist Borel-partitions $\cR_n$ of $\Gb$ such that the $(\Gb, \cR_n)$ locally converge to $(\Gb^*, \cP^*)$. By Claim~\ref{cl:coloured_conv} applied to the sequence $(\Gb_n,\cP_n)$, we get that $a=\varphi_{\Gb^*}/\cP^*$. Let $b_n = \varphi / \cR_n$. Then, $b_n \in Q_k$ and applying the claim to $(\Gb, \cR_n)$ gives $b_n \to a$ as well. Thus we get $a \in \overline{Q}_k$ as stated, finishing the proof of the theorem.
\end{proof}

With the help of Theorem~\ref{thm:local_to_local-global_boost}, we are ready to prove Theorem~\ref{thm:intro_global-conv}.

\begin{proof}[Proof of Theorem~\ref{thm:intro_global-conv}.] Let $f(\Gb)=\rho_{\Gb}(E(\Gb))$ be the graphing parameter that assigns to every graphing the total rank of its edge set. Then $\varphi_{\Gb,f}=\rho_{\Gb}$, and Theorems~\ref{thm:lovasz} and \ref{thm:local_to_local-global_boost} together complete the proof. 
\end{proof}

%%%%%%%%%%%%%%%%
\subsection{Non-continuity for local convergence}
\label{sec:noncon}
%%%%%%%%%%%%%%%%

It was proved in~\cite[Theorem~4.7]{lovasz2023matroid} that the rank of the whole edge set of a graph is locally estimable, i.e.\ continuous with respect to local convergence~\cite[Theorem 22.3]{lovasz2012large}. We show by an example that local convergence does not guarantee quotient-convergence of the matroids.

\begin{ex}\label{ex:loc}
Let $(G_1,G_2,\dots)$ be a sequence of $3$-regular, $c$-expander graphs with $|V(G_n)|=2n$ and $0<c<1$, such that their girth $g_n$ tends to infinity. Let $H_n=G_n \cup G_n$, that is, $H_n$ is the disjoint union of two copies of $G_n$. Clearly, the interlaced sequence $G_1,H_1,G_2,H_2,\dots$ is locally convergent (it converges to any 3-regular treeing). Let $\rho_{G_n}=r_{G_n}/|V(G_n)|$, and similarly for $H_n$. We claim that $Q_2(\rho_{G_n})$ is far from $Q_2(\rho_{H_m})$ if $n$ and $m$ are large, so the sequence $(\rho_{G_1}, \rho_{H_1}, \rho_{G_2}, \rho_{H_2},\dots)$ is not quotient-convergent.

The partition of $H_n$ into two copies of $G_n$ defines a quotient of the cycle matroid: 
    \[
    a_n \in Q_2(H_n) \subseteq \R^{2^2}, \ a_n(\emptyset)=0, \ a_n(\{1\})=a_n(\{2\})=\frac{2n-1}{4n}, \ a_n(\{1,2\})=\frac{4n-2}{4n}.
    \] 
The limit of $a_n$ is 
    \[a \in \R^{2^2}, \ a(\emptyset)=0, \ a(\{1\})=a(\{2\})=\frac12, \ a(\{1,2\})=1.\]
We show that if $n$ is sufficiently large, then $G_n$ cannot have $2$-quotients that are $\eps$-approximations of $a$, where $\eps=c/32$. This will imply that $d^\Haus(Q_2(\rho_{G_n}), Q_2(\rho_{H_m}))>\eps$. We choose $n$ large enough so that $g_n>32/c$.
    
By the definition of expander sequences, for every $S\subseteq V(G_n)$ with $|S|\le n$, the number of edges connecting $S$ to $V(G_n)\setminus S$ is at least $c\cdot |S|$. Let $E(G_n)=E_1\cup E_2$ be any 2-partition of the edges. For $i\in\{0,1,2,3\}$, we denote by $V_i$ the set of vertices of $G_n$ with degree $i$ in $E_1$ and let $m_i=|V_i|$. Suppose that this partition results in a quotient closer to $a$ than $\eps$. Then, $|\rho_{G_n}(E_i)-\frac12|<\eps$ for $i=1,2$, and so $|r_{G_n}(E_i)-n|<2\eps n$. Let us note that $r_{G_n}(E_1) \leq 2n - m_0$, and the previous two inequalities imply $m_0\le (1+2\eps)n$. The same bound holds for $m_3$.
    
We bound the number of $E_1$-components as follows. The number of singletons is clearly $m_0$. If an $E_1$-component is not a tree, then it has at least $g_n$ vertices, and hence the number of such components is at most $2n/g_n$. Finally, if a component is a non-singleton tree, then it has at least $2$ vertices of degree $1$, so the number of such components is bounded by $m_1/2$. Hence
\[
m_0+\frac{m_1}{2}\ge 2n - r_{G_n}(E_1)-\frac{2n}{g_n}  \ge \Big(1 - 2\eps-\frac2{g_n}\Big) n \ge \Big(1 - \frac{c}{8}\Big) n. 
\]    
Adding the analogous inequality for $E_2$, we get
\begin{equation*}\label{EQ:EXPAND0}
2m_0+m_1+m_2+2m_3 \ge  \Big(2-\frac{c}{4}\Big)n,
\end{equation*}
and using that $m_0+m_1+m_2+m_3=2n$, we obtain $m_1+m_2 \le cn/4$. Without loss of generality, we may assume that $m_0\le m_3$, implying that $m_0\le |V(G_n)|/2$. The neighbours of $V_0$ in $G$ are contained in $V_1\cup V_2$, and hence by the expansion of $G_n$, we have $m_1+m_2\ge c m_0$. So, 
\[
m_0+m_1+m_2+m_3 \le \frac{c+1}{c}(m_1+m_2)+m_3 \le  \frac{c+1}{c}\,\frac{c}{4}\,n+ \Big(1+\frac{c}{16}\Big)n = \Big(1+\frac{5c}{16}\Big)n <2n .
\]
This is a contradiction.
\end{ex}

%%%%%%%%%%%%%%%%
\section{Exposed minorizing measures and hyperfinite spanning forests}
\label{sec:extreme}
%%%%%%%%%%%%%%%%

In this section, we focus on the set of minorizing finitely additive measures of the rank function of a graphing. In Section~\ref{sec:rn}, we establish a connection between the Radon-Nikodym derivatives of such measures and hyperfinite essential spanning forests. We characterize the exposed points of minorizing finitely additive measures in Section~\ref{sec:exposed}, and then introduce the notion of weakly exposed points in Section~\ref{sec:weakly}. Based on these observations, we prove Theorems~\ref{thm:exposed} and~\ref{thm:exposed_char} in Section~\ref{sec:mainthms}. Finally, in Section~\ref{sec:submod}, we give a conceptually new proof of the submodularity of the rank function $\rho_\Gb$ of a graphing $\Gb$ that might be of independent combinatorial interest.

%%%%%%%%%%%%%%%%
\subsection{Measures, functions and subsets}
\label{sec:rn}
%%%%%%%%%%%%%%%%

Let $\Gb =(J, \mu,E)$ be a graphing and $\rho_\Gb$ be the rank function of its cycle matroid. If $\Gb$ has no finite components, then $\rho_\Gb(E)=1$. Since any $\alpha \in \mmp(\rho_\Gb)$ is a $\sigma$-additive measure, it admits a Radon-Nikodym derivative with respect to $\eta_\Gb $, denoted by
\[f_{\alpha}(x) = \frac{\partial \alpha}{\partial \eta_\Gb } (x).\]
For a function $f\in L^1(\eta_\Gb)$, let $\eta_f$ denote the measure 
\[
\eta_f(A)=\int_A f\diff \eta_\Gb.
\]
By the above, we can identify the convex sets of measures $\mmp(\rhog)$ and $\bmmp(\rho_\Gb)$ with convex sets of measurable functions.  By a slight abuse of notation, we say that a function $f\in L^1(\eta_\Gb)$ is in $\mmp(\rhog)$ or $\bmmp(\rho_\Gb)$ if the corresponding measure $\eta_f$ is in $\mmp(\rhog)$ or $\bmmp(\rho_\Gb)$, respectively. Note that $f_{\alpha}$ takes values in $[0,1]$. 

If $f_{\alpha}$ is $0$-$1$ valued $\eta_\Gb$-almost everywhere, then $\alpha$ is clearly extreme in both $\mmp(\rho_\Gb)$ and $\bmmp(\rhog)$. In such a case, let us consider the corresponding measurable set of edges $E_{\alpha} = \{e \in E_\Gb \mid f_\alpha(e)=1\}$. Our first observation  is that $\Tb=(J,\mu,E_{\alpha})$ is a hyperfinite essential spanning forest.  

\begin{lem}\label{lem:hyperfinite_is_minorized}
    Let $G=(J,\mu,E)$ be a graphing and $\Tb=(J,\mu,F)$ be a spanning subgraphing. The following are equivalent:
    \begin{enumerate}[label=(\roman*)]\itemsep0em
        \item $\Tb$ is a hyperfinite spanning forest; \label{him:i}
        \item the measure $\alpha_F(A) = \eta_\Gb  (A \cap F)$ is in $\mmp(\rho_\Gb)$.\label{him:ii}
    \end{enumerate}
    Similarly, the following are equivalent:
\begin{enumerate}[label=(\roman*')]\itemsep0em
        \item $\Tb$ is a hyperfinite essential spanning forest; \label{him:i'}
        \item the measure $\alpha_F(A) = \eta_\Gb  (A \cap F)$is in $\bmmp(\rho_\Gb)$.\label{him:ii'}
    \end{enumerate}
\end{lem}
\begin{proof}\mbox{}\\
$\ref{him:i}\Rightarrow\ref{him:ii}$: Let $B$ be a Borel subset of $E$. Then $B\cap F$ is clearly a hyperfinite forest, so $\alpha_F(B)=\eta_\Gb(B\cap F)=\rho_{\Gb}(B\cap F)\le\rho_\Gb(B)$, where the first equality is by the definition of $\alpha_F$, the second equality holds by Lemma~\ref{lem:levitt}, and the inequality follows from \cite[Lemma 4.3]{lovasz2023matroid}.\\
$\ref{him:ii}\Rightarrow\ref{him:i}$: From $\alpha_F\le\rho_\Gb\le\eta_\Gb$ it follows that $\eta_\Tb\le \rho_\Gb|_\Tb=\rho_\Tb\le \eta_\Tb$. Therefore, we get $\eta_\Tb=\rho_\Tb$, implying that $\Tb$ is a hyperfinite spanning forest by Lemma~\ref{lem:levitt}.\\
$\ref{him:i'}\Rightarrow\ref{him:ii'}$: By the equivalence of $\ref{him:i}$ and $\ref{him:ii}$, we have $\alpha_F\in \mmp(\rho_\Gb)$. By Corollary~\ref{cor:levitt}, we have $\alpha_F(E)=\eta_\Gb(F)=\rho_\Gb(E)$, showing $\alpha_F\in\bmmp(\rho_\Gb)$.\\
$\ref{him:ii'}\Rightarrow\ref{him:i'}$: By the equivalence of $\ref{him:i}$ and $\ref{him:ii}$, $\Tb$ is a hyperfinite forest, hence $\rhog(F)=\eta_\Gb(F)$. Since $\eta_\Gb(F)=\alpha_F(E)=\rhog(E)$, we get $\rhog(E)=\eta_\Gb(F)$, implying that $\Tb$ is essential by Corollary~\ref{cor:levitt}.    
\end{proof}

%%%%%%%%%%%%%%%%
\subsection{Exposed points}
\label{sec:exposed}
%%%%%%%%%%%%%%%%

Let $C$ be a compact, convex subset of a locally convex topological space. We say that $v\in C$ is exposed if there exists $\varphi$ linear, continuous functional such that $v$ is the unique maximizer of $\varphi$ on $C$, that is,  $\varphi(v) > \varphi(u)$  for any $u \in C$, $u \neq v$. In this case, we say that $v$ is {\it exposed} by the function $\varphi$ and call $\varphi$ a {\it witness function} for $v$. It is not difficult to see that exposed points are extreme. In this section we characterize the exposed points of $\mmp(\rhog)$ and $\bmmp(\rhog)$. We prove that all these points are of the form $\alpha_F$ where $F$ is the edge set of some hyperfinite spanning forest, or in other words, all the exposed $f_{\alpha}$ are $0$-$1$ valued.

While being extreme is a purely algebraic property, being exposed my depend on the underlying topology. Therefore, first we need to specify which space to embed $\mmp(\rhog)$ and $\bmmp(\rhog)$ into. A natural choice would be using $L^{p}$ with the weak topology for some $p\in[1,\infty)$, or $L^{\infty}$ with the weak$^*$ topology. As it turns out, the choice of the space does not matter as $\mmp(\rhog)$ and $\bmmp(\rhog)$ are compact and the exposed points are exactly same in all cases. For ease of discussion, we use the space $L^1$ with the weak topology.

\begin{thm} \label{thm:exposed_integer}
    Let $f_\alpha \in \bmmp(\rho_\Gb)$ be exposed by a function $g\in L^\infty(E, \eta_\Gb )$. Then, $f_{\alpha}$ is $0$-$1$ valued $\eta_\Gb $-almost everywhere.
\end{thm}
\begin{proof}
    We treat $g$ as a weight function on the edges, and our goal is to find a maximum weight hyperfinite essential spanning forest $F$ with respect to this weight function. Recall that $\alpha_{F}$ denotes the measure in $\bmmp(\rhog)$ defined as $\alpha_F(A) = \eta_\Gb (A \cap F)$ for all measurable $A \subseteq E$. We show that there exists an $F$ satisfying $\alpha_{F}(\{g \geq t\}) = \rhog (\{g \geq t\})$ for every $t\in\mathbb{R}$. For any $\beta \in \bmmp(\rhog)$, in a layer-cake integration this implies 
    \[\int_{E} \vmathbb{1}_{F} \cdot g \diff \eta_\Gb  = \int_E g \diff \alpha_F  \geq \int_E g \diff \beta = \int_{E} f_{\beta} \cdot g \diff \eta_\Gb .\] 
    By the assumption that $f_{\alpha}$ is the unique maximizer, this proves $f_{\alpha}=\vmathbb{1}_{F}$ in $L^1(E, \eta_\Gb )$.

    To verify the existence of a required $F$, we run a variant of Kruskal's greedy algorithm~\cite{kruskal1956shortest} on $\Gb $, sometimes called Prim's invasion algorithm. As a tie-breaking rule, we fix an arbitrary Borel ordering of $E$ in advance, and whenever there is a tie among the $g$ values of finitely many edges considered, we choose the maximal one according to the Borel ordering. The algorithm constructs Borel edge sets $F_1 \subseteq F_2 \subseteq \dots$, with $F= \bigcup F_n$ being the final output. In the first step, every vertex simultaneously chooses the incident edge with the highest $g$-value to be included in $F_1$. During the $n^{\mathrm{th}}$ step, every finite component of $F_{n-1}$ adds the edge with the highest $g$-value from the set of edges leaving it if such an edge exists, and thus we obtain $F_n$. In case there is no edge leaving a finite component, observe that it is also a component in theoriginal graphing and $F_{n-1}$ is already a spanning tree in it. Also note that infinite components do not add any edges, but it may happen that a finite component adds an edge that connects it to an infinite one. 

    It is not difficult to see that the sets $F_n$ are Borel, cycles cannot be created, and $F$ is essential. 
    We claim that $\eta_\Gb (F) = \rho_{\Gb}(E)$, so $F$ is hyperfinite by Theorem~\ref{lem:levitt}. Indeed, set up a payment function as follows. In the beginning, every vertex receives $1$ token. In a general step, every finite connected component gives $1$ token to the edge it chooses. After the addition of the edges, for every finite connected component there is exactly one edge who received $2$ tokens. This edge gives $1$ token back to the component, thus ensuring that every finite component has a token at the beginning of the next step. Should a finite $F$-component become spanning in a finite $\Gb$-component, and therefore has no use for its token, this token is redistributed equally among the vertices of this component. In the end, vertices in finite $\Gb$-components paid $1-\frac{1}{|\Gb_{x}|}$ token, and vertices in infinite $\Gb$-components paid 1 token. On the other hand, every edge in $F$ received 1 token, hence $\rho_{\Gb}(E)=\eta_\Gb (F)$ by the Mass Transport Principle. As $F$ is hyperfinite, by Lemma~\ref{lem:hyperfinite_is_minorized}, we have $\alpha_F \in \bmmp(\rhog)$. 
    
    Let $A_t = \{g \geq t\}$ and $F_{\geq t} = F \cap A_t$. Observe that $F_t$ is an essential spanning forest of $A_t$. That is, for $\mu$-almost every $x \in J$,
    \begin{enumerate}[label=(\arabic*)]\itemsep0em
        \item \label{item:finite} if the $A_t$-component of $x$ is finite, then $F_{\geq t}$ is a spanning tree inside this component, and
        \item \label{item:infinite} if the $A_t$-component of $x$ is infinite, then its $F_{\geq t}$-component is also infinite -- though it might be smaller.
    \end{enumerate}    
    For \ref{item:finite}, note that the edges leaving the $A_t$-component all have $g$-values strictly less than $t$, while being an $A_t$-component means that it can be connected by edges with $g$-value at least $t$. So the algorithm will connect these points before growing the component any further. For \ref{item:infinite}, the reasoning is analogous. The conclusion, however, is different, This is due to the fact that we cannot guarantee the connectedness of $F_{\geq t}$ inside the $A_t$-component, because the $F_{\geq t}$-components stop growing once they are all infinite. 

    Observations \ref{item:finite} and \ref{item:infinite} imply $\alpha_F(A_t) = \rho_\Gb (A_t)$, completing the proof of the theorem.
\end{proof}

The algorithm also proves the following corollary.

\begin{cor}\label{COR:FOREST-EXTEND}
Let $\Gb$ be a graphing and $F\subseteq E(\Gb)$ be a hyperfinite subforest of $\Gb$. Then $F$ can be extended to a hyperfinite essential spanning forest of $\Gb$.
\end{cor}
\begin{proof}[Proof of Corollary~\ref{COR:FOREST-EXTEND}]
    Take any Borel ordering of the set $E(\Gb)\setminus F$ and run the same invasion algorithm as in the proof of Theorem~\ref{thm:exposed_integer} with the only difference that the starting graph, which was the empty graph in the previous case, is now $F$. 
\end{proof}

The exposed points of $\mmp(\rhog)$ can be characterized in an analogous way.

\begin{thm}\label{thm:exposed_I_integer}
        Let $f_\alpha \in \mmp(\rho_\Gb)$ be exposed by a function $g\in L^\infty(E, \eta_\Gb )$. Then, $f_{\alpha}$ is $0$-$1$ valued $\eta_\Gb $-almost everywhere.
\end{thm}
\begin{proof}
    Run the algorithm from the proof of Theorem~\ref{thm:exposed_integer}. That gives us a hyperfinite essential spanning forest $F$ with the property that for all $t\in\R$, we have $\alpha_F(\{g\ge t\})=\rho_\Gb(\{g\ge t\})$. Now let $F'=\{e\in F\mid  g(e)\ge 0\}$. We claim that $F'$ maximizes the function $g$ on $C$. The proof, again, relies on the cake-layer representation. Take an $f\in \mmp(\rhog)$. We may assume that $f$ is zero on the set $N=\{g<0\}$, otherwise $\varphi_g(f)<\varphi_g(f\one_{N^c})$.
    \begin{gather*}   
    \varphi_g(f)=\int_{N^c}g\cdot f\diff\eta_\Gb=
    \int_{N^c}g\diff\eta_f=
    \int_0^{\infty}\eta_f\left(\{g\ge t\}\right) \diff\lambda(t)\le \\ 
    \int_0^{\infty}\rho_\Gb\left(\{g\ge t\}\right) \diff\lambda(t)=
    \int_0^{\infty}\alpha_{F'}\left(\{g\ge t\}\right) \diff\lambda(t)=\varphi_g(\one_{F'}).
   \end{gather*}
 By the assumption that $f_{\alpha}$ is the unique maximizer, this proves $f_{\alpha}=\vmathbb{1}_{F'}$.
    \end{proof}

Let us make a couple of remarks.

\begin{rem}
    If $f_{\alpha}$ is $0$-$1$ valued, then it is exposed -- the witness for this is simply choosing $g=f_{\alpha}$. In fact, in this case $f_\alpha$ is \emph{strongly exposed}, i.e.\ for any sequence $\beta_n \in \bmmp(\rhog)$ we have $\varphi_{g}(f_{\beta_n}) \to \varphi_g(f_{\alpha}) \Rightarrow ||f_{\beta_n} - f_{\alpha}||_p \to 0$ for all $p \in [1,\infty)$. That is, Theorem~\ref{thm:exposed_integer} completely characterizes exposed points of $\bmmp(\rhog)$, and a similar statement holds for $\mmp(\rhog)$ as well.
\end{rem}

\begin{rem} \label{rem:exposed_are_dense}
    By the Klee--Straszewicz theorem~\cite[Theorem 2.1]{klee1958extremal}, exposed points are dense among extreme points of $\bmmp(\rhog)$ in the weak$^*$ topology of $L^{\infty}(E, \eta_\Gb )$. Moreover, the exposed points determine $\bmmp(\rhog)$ as $\bmmp(\rhog)=\overline{\conv(\exp(\bmmp(\rhog)))}^{w^*}$. Again, a similar statement holds for $\mmp(\rhog)$ as well.
\end{rem}

%%%%%%%%%%%%%%%%
\subsection{Weakly exposed points}
\label{sec:weakly}
%%%%%%%%%%%%%%%%

Even in finite dimension, extreme points are not necessarily exposed. To see this, consider the following set in $\bR^2$:
\[S=\big\{(x,y)\mid -1\leq x\leq 1,\ -1\leq y\leq 1\big\}\cup\big\{(x,y)\mid x^2+y^2\leq 1\big\}.\]
Then $(-1,0)$ and $(1,0)$ are clearly extreme points. However, there is a unique supporting hyperplane through any of them: $x=-1$ through $(-1,0)$ and $x=1$ through $(1,0)$, hence these points are not exposed. Nevertheless, these extreme points can be exposed in finitely many steps by iteratively restricting the problem to the intersection of the convex set and a supporting hyperplane. In this subsection, we use this idea to generalize Theorem~\ref{thm:exposed_integer} to extreme points that we call \emph{weakly exposed}.

We say that an extreme point $x \in \bmmp(\rhog)$ is \emph{exposed in two steps} by $L^1$ functions, if one can find $g_1, g_2 \in L^1(E, \eta_\Gb )$ such that $\varphi_{g_1}(x)$ is maximal in $\bmmp(\rhog)$, and $x$ strictly maximizes $\varphi_{g_2}$ in $\{y \in \bmmp(\rhog) \mid \varphi_{g_1}(y) = \varphi_{g_1}(x)\}$. Note that the proof of Theorem~\ref{thm:exposed_integer} yields the same conclusion for points that are exposed in two steps. The only difference is that during the algorithm, we always choose an edge that is maximal with respect to $g_1$, and in case of ties, we choose the one maximal with respect to $g_2$. Finally, if there is still a tie, we decide according to the Borel ordering we fixed at the beginning. 

More generally, let $\alpha\in\omega_1$ be a countable ordinal and let $\mathbf{g}\colon \alpha\to L^1(E, \eta_\Gb )$ be a list of functions. Let us define $\mathbf{g}(x)\in \R^{\alpha}$ by $f(x)_{\beta}=\mathbf{g}(\beta)(x)$. Then, $\mathbf{g}$ defines a Borel partial ordering on $\bmmp(\rhog)$ based on the lexicographical ordering of the images of points of $\bmmp(\rhog)$ in $\R^{\alpha}$. We denote this ordering by $<_{\mathbf{g}}$ and call a point $x\in \bmmp(\rhog)$ \emph{exposed in $\alpha$ steps} by the list $\mathbf{g}$ if $x$ is the unique maxima of $<_{\mathbf{g}}$ in $\bmmp(\rhog)$. In other words, for all $y\in \bmmp(\rhog)$ there exists an ordinal $\beta_y\le \alpha$ such that for all $\gamma\leq \beta_y$, we have $\mathbf{g}(\gamma)(x)=\mathbf{g}(\gamma)(y)$ and  $\mathbf{g}(\beta_y)(x)>\mathbf{g}(\beta_y)(y)$. We say $x\in \bmmp(\rhog)$ is \emph{weakly exposed}, if it is exposed in $\alpha$ steps for some countable ordinal $\alpha$.

\begin{thm}\label{thm:weakly}
    Let $f_\alpha \in \bmmp(\rhog)$ be weakly exposed by a list of functions in $L^1(E, \eta_\Gb )$. Then, $f_{\alpha}$ is $0$-$1$  valued $\eta_\Gb $-almost everywhere.
\end{thm}
\begin{proof}
The proof of Theorem~\ref{thm:exposed_integer} holds for weakly exposed points as well by always choosing the $<_{\mathbf{g}}$-maximal edge.
\end{proof}

\begin{cor}\label{cor:weakstrong}
    Weakly exposed points of $\bmmp(\rhog)$ are strongly exposed.
\end{cor}

In finite dimension, all extreme points of a compact convex set are weakly exposed. In infinite dimension, however, this is not necessarily true. The closed unit ball of $\ell^p(\N)$ for $1<p<2$ in $\ell^2(\N)$ has extreme points that are not even support points, see~\cite[Example 7.5]{simons2007hahn}.

\begin{rem} \label{rem:WmaxSF}
The algorithm in the proof of Theorem~\ref{thm:exposed_integer} results in a Wired Maximal Spanning Forest with respect to an arbitrary Borel ordering $<$ on the edge set, as defined in the introduction. Usually, such forests are studied with the ordering coming from iid $[0,1]$ labels on the edges of a countable graph.    
\end{rem}

%%%%%%%%%%%%%%%%
\subsection{Characterizations of independent sets and bases}
\label{sec:mainthms}
%%%%%%%%%%%%%%%%

In this subsection we prove Theorems~\ref{thm:exposed} and~\ref{thm:exposed_char}.

\begin{proof}[Proof of Theorem~\ref{thm:exposed}]
First we prove~\ref{exposed:a}. For the `if' direction, note that $\eta_\Gb|_F\in\mmp(\rho_\Gb)$ by Lemma~\ref{lem:hyperfinite_is_minorized}. The Radon-Nikodym derivative of the measure $\eta_\Gb|_F$ is $\one_F$. The function $g=\one_F-\one_{F^c}$ is a witness function for $\alpha$ being exposed. To see the `only if' direction, using Theorem~\ref{thm:exposed_I_integer}, we know that exposed functions are $0-1$ valued. Therefore, the exposed points are measures of the form $\eta_\Gb|_F$ for some subset $F\subseteq E$. However, the subsets for which this restriction is in $\mmp$ are exactly the hyperfinite subforests by Lemma~\ref{lem:hyperfinite_is_minorized}.

Now we prove~\ref{exposed:b}. For the `if' direction, note that $\eta_\Gb|_F\in\bmmp(\rho_\Gb)$ by Lemma~\ref{lem:hyperfinite_is_minorized}. The Radon-Nikodym derivative of the measure $\eta_\Gb|_F$ is $\one_F$. The function $g=\one_F-\one_{F^c}$ is a witness function for $\alpha$ being exposed. To see the `only if' direction, using Theorem~\ref{thm:exposed_I_integer}, we know the exposed functions are $0-1$ valued. Therefore, the exposed points are measure of the form $\eta_\Gb|_F$ for some subset $F\subseteq E$. However, the subsets for which this restriction is in $\bmmp(\rhog)$ are exactly the hyperfinite essential spanning forests by Lemma~\ref{lem:hyperfinite_is_minorized}.
\end{proof}

\begin{proof}[Proof of Theorem~\ref{thm:exposed_char}]\mbox{}\\
$\ref{indepi}\Leftrightarrow\ref{indepiii}$: The equivalence follows by Lemma~\ref{lem:levitt}.\\ 
$\ref{indepi}\Rightarrow\ref{indepii}$: The implication follows by Corollary~\ref{COR:FOREST-EXTEND}.\\
$\ref{indepii}\Rightarrow\ref{indepi}$: The implication follows from the fact that every Borel subset of a hyperfinite treeing is a hyperfinite treeing.\\
$\ref{basei}\Leftrightarrow\ref{baseii}$ The equivalence follows by Corollary~\ref{cor:levitt}. 
\end{proof}

%%%%%%%%%%%%%%%%%
\subsection{Independent sets and exchange properties}
\label{sec:submod}
%%%%%%%%%%%%%%%%%

The setfunction $\rhog$ was shown to be submodular in~\cite{lovasz2023matroid}. The proof relies on structural analysis of so-called $\mathcal{B}$-partitions and is somewhat complicated. In this subsection we give a new proof of this result that conceptually differs from the original one. The main idea is to define a submodular setfunction through independent sets, and then to show that the function thus obtained is identical to $\rhog$.

Let $\Gb=(J,\mu,E)$ be a graphing. We define $\cI(\Gb)=\{F\subseteq E\mid F\text{ is a hyperfinite forest}\}$ and $\cB(\Gb)=\{F \subseteq E\mid F\text{ is a hyperfinite essential spanning forest}\}$. By Corollary~\ref{COR:FOREST-EXTEND}, $\cI(\Gb)$ is the set of subsets of the elements of $\cB(\Gb)$, in other words, 
$$ 
\cI(\Gb)=\{F\mid  F\subseteq F'\ \text{for some}\ F'\in \cB(\Gb)\}.
$$
By Lemma~\ref{lem:levitt}, the measure of any hyperfinite essential spanning subforest is the same.

Let $(J,\mu)$ be a standard Borel space. We say that family of subsets $\cI$ of $J$ is the {\it independence family of a measurable matroid} with respect to the measure $\mu$, if
\begin{enumerate}[label=(I\arabic*)]\itemsep0em
    \setcounter{enumi}{-1}
    \item $\cI$ is closed under taking subsets, \label{it:i0}
    \item if $I_1\subseteq I_2\subseteq\dots $ with $I_n\in \cI$ for all $n\in\bZ_+$, then $\cup I_n\in\cI$,\label{it:i1}
    \item for all $I_1,I_2\in \cI$ with $\mu(I_1)<\mu(I_2)$, there exists $I_3\in\cI$ such that $I_1\subseteq I_3\subseteq I_1\cup I_2$ and $\mu(I_3)>\mu(I_1)$,\label{it:i2}    
\end{enumerate}
In certain cases, the following reformulation of~\ref{it:i2} is easier to work with.
\begin{enumerate}[label=(I\arabic*')]\itemsep0em
    \setcounter{enumi}{1}
    \item for all $I_1,I_2\in \cI$, there exists $I_3\in\cI$ such that $I_1\subseteq I_3\subseteq I_1\cup I_2$ and $\mu(I_3)\ge\mu(I_2)$.\label{it:i2'}
\end{enumerate}

\begin{lem}\label{lem:equiv}
    $\{\ref{it:i1},\ref{it:i2}\}$ and $\{\ref{it:i1},\ref{it:i2'}\}$ are equivalent.
\end{lem}
\begin{proof}
    It is clear that \ref{it:i2'} implies \ref{it:i2}. It remains to show that \ref{it:i1} and \ref{it:i2} together imply \ref{it:i2'}. We build a transfinite sequence of Borel sets $A_\alpha$ ($\alpha\leq\omega_1$). Let $A_0=I_1$. Using \ref{it:i2}, there exists $A_1\in\cI$ such that $A_0\subseteq A_1\subseteq A_0\cup I_2$. Similarly, there exists $A_2\in\cI$ such that $A_1\subseteq A_2\subseteq A_1\cup I_2=A_0\cup I_2$. In general, let $f\colon\cI\to\cI$ be a function with the following properties: $A\subseteq f(A)\subseteq I_2\cup A$ for all $A\in\cI$, $\mu(A)<\mu(f(A))$ whenever $\mu(A)<\mu(I_2)$, and $A=f(A)$ whenever $\mu(A)=\mu(I_2)$. Note that such a function exists by~\ref{it:i2}. Now do this recursion transfinitely, up to ordinal $\omega_1$.

    If $\mu(A_\alpha)=\mu(I_2)$ for some $A_\alpha$ in the sequence, then we are done. Suppose that $\mu(A_{\alpha})<\mu(I_2)$ for all $A_\alpha$. Since there exists no increasing sequence of real numbers of order $\omega_1$, the sequence $A_{\alpha}$ has to stabilise at a certain point. Call the first ordinal with such an independent set $\alpha_0$; note that $A_{\alpha_0}$ is in $\cI$ by~\ref{it:i1}. That is, $f(A_{\alpha_0})=A_{\alpha_0}$, but $\mu(A_{\alpha_0})<\mu(I_2)$, contradicting the choice of $f$.
%     \kristof{Innen másik bizonyítás kezdődik.}

%     $$A_\alpha=\begin{cases}
%         \bigcup\limits_{\beta<\alpha} A_{\beta},\quad&\text{ if $\alpha$ is a limit-ordinal},\\
%         f(A_{\alpha^-}),\quad&\text{ if $\alpha$ is a successor ordinal},
%     \end{cases}$$
%     where $\alpha^-$ is the predecessor of $\alpha$. This recursion gives us a family of independent sets $(A_{\alpha})_{\alpha\in\omega_1}$.

%     A proof without transfinite recursion: For a Borel set $B$ let $f(B)$ be a Borel set with the following properties: $B\subseteq f(B)$, $f(B)\in\cI$ and $$\mu(f(B))\ge\frac{1}{2}\left(\mu(B)+\sup_{\substack{B\subseteq C\subseteq B\cup I_2\\ C\in\cI}}f(C) 
%  \right).$$ 
% Now let $B_0=I_1$, $B_1=f(B_0)$, $B_{n+1}=f(B_n)$. Let $B_\omega=\cup B_n$. By (I1), $B_\omega\in\cI$. Now assume $\mu(B_\omega)<\mu(I_2)$. Then let $B'=f(B_\omega)$, and let $a=\mu(B')-\mu(B_\omega)$. Also, let $n\in\N$ be such that $\mu(B_\omega)-\mu(B_n)<a$. Therefore $B_n$ has an extension $C=B'$ such that $\mu(C)>\mu(B_n)+2a$. However, $\mu(B_{n+1})\leq\mu(B_\omega)<\mu(B_n)+a$. However this contradicts the fact that $\mu(B_{n+1})\ge\frac{1}{2}(\mu(B)+\mu(C))$. Therefore $\mu(B_\omega)\ge \mu(I_2)$.
\end{proof}

\begin{lem}\label{lem:submod}
    Let $\cI$ be the family of independent subsets of a measurable matroid with respect to measure $\mu$. Then the setfunction $r(A)=\sup\{\mu(I)\mid  I\in\cI,\ I\subseteq A\}$ is submodular.
\end{lem}
\begin{proof}
    We here imitate the proof of \cite[Theorem 5.12]{lovasz2023submodular}.
    Let $A,B$ Borel sets and $\varepsilon>0$. Then there exists $I_1,I_2\in\cI$ such that $I_1\subseteq A\cap B,I_2\subseteq A\cup B$, and $r(A\cap B)-\varepsilon\leq \mu(I_1)$, $r(A\cup B)-\varepsilon\leq \mu(I_2)$. By the exchange property~\ref{it:i2'}, there exists a set $H\in\cI$ such that $I_1\subseteq H\subseteq I_1\cup I_2\subseteq A\cup B$ and $\mu(H)\ge \mu(I_2)\ge\mu(A\cup B)-\varepsilon$. Then $\mu(H\cap (A\cap B))\ge\mu(I_1\cap(A\cap B))=\mu(I_1)\geq r(A\cap B)-\varepsilon$.
    Therefore, using~\ref{it:i0}, we get
    \begin{align*}
    r(A)+r(B)
    {}&{}\ge 
    r(H\cap A)+r(H\cap B)\\
    {}&{}=
    \mu(H\cap A)+\mu(H \cap B)\\
    {}&{}=
    \mu(H\cap (A\cap B))+\mu(H\cap (A\cup B))\\
    {}&{}\geq 
    r(A\cap B)-\varepsilon+\mu(A\cup B)-\varepsilon.
    \end{align*}
    As this inequality holds for all $\varepsilon>0$, it follows that $r$ is submodular.
\end{proof}

With the help of Lemma~\ref{lem:submod}, first we show that $\cI(\Gb)$ defines a submodular function.

\begin{lem}\label{lem:exchange}
    Let $\Gb$ be a graphing. Then $\cI(\Gb)$ forms the family of independent sets of a measurable matroid with respect to the edge measure $\eta_\Gb$.
\end{lem}
\begin{proof}
It is not difficult to check that Properties~\ref{it:i0} and~\ref{it:i1} are satisfied, hence we concentrate on \ref{it:i2'}. Let $I_1,I_2\in\cI(Gb)$ and consider the graphing $\Hb=(V,\mu,I_1\cup I_2)$. Then $\eta_{\Hb}=\eta_{\Gb}|_{I_1\cup I_2}$. Let $I_3$ be an extension of $I_1$ such that $I_3$ is a hyperfinite essential spanning forest in $\Hb$; such an extensions exists by Corollary~\ref{COR:FOREST-EXTEND}. Then $\eta_\Gb(I_3)=\eta_\Gb(I_1\cup I_2)\ge \eta_\Gb(I_2)$, therefore the exchange property~\ref{it:i2'} holds.
   \end{proof}

Combining the above results, now we are ready to prove that $\rho_{\Gb}$ is submodular.

\begin{thm}
    Let $\Gb=(J,\mu,E)$ be a graphing. Then the setfunction $\rho_{\Gb}$ defined by \eqref{eqn:graphing_rank} is submodular.
\end{thm}
\begin{proof}
    Every graphing contains a hyperfintie essential spanning forest. Furthermore, by Corollary~\ref{cor:levitt}, the edge measure of a hyperfinite essential spanning forest is exactly the rank of the graphing. From these observations it follows that for any $A\subseteq E$, $\rho_\Gb(A)=\sup\{\eta_\Gb(F)\mid F\subseteq A\ \text{is a hyperfinite forest in}\ (J,\mu,A)\}$. By Lemma~\ref{lem:exchange}, hyperfinite subforests form the family of independents sets of a measurable matroid with respect to the edge measure $\eta_\Gb$. Therefore, the theorem follows by Lemma~\ref{lem:submod}.
\end{proof}
    
%%%%%%%%%%%%%%%%
\section{Duality} 
\label{sec:duality}
%%%%%%%%%%%%%%%%

In this section, we present various observations based on the idea that in an infinite planar graph, the complement of an essential spanning forest is an essential spanning forest of the dual.

Roughly speaking, a \emph{planar graphing} is a graphing $\Gb=(J,\mu,E)$ whose connected components are planar and are also endowed with a concrete embedding in a measurable way. The embedding is encoded by specifying the finite facial cycles of the components, and we assume this collection of edge sets to be a Borel set inside the standard Borel space of finite subsets of $E$. For the precise definition, see~\cite[Section 3.2]{conley2021one}. The component of a $\mu$-random vertex, together with the embedding is a \emph{unimodular random map} in the sense of~\cite{angel2018hyperbolic}.

We will, in fact, restrict our attention to the case when all facial cycles of $\Gb$ are finite, and refer to them as \emph{faces}. The assumption of $2$-vertex-connectivity and bounded size of faces mentioned in the Introduction also applies. The fact that the encoding is Borel ensures that we can talk about the dual graphing $\Gb^*=(J^*, \mu^*,E^*)$, where $J^*$ is the set of faces of $\Gb$. There is a measure preserving Borel bijection $\sigma\colon E \to E^*$ that denotes the usual correspondence: $\sigma(e)$ connects two distinct faces $x_1^*$ and $x_2^*$ if and only if $e$ is a common boundary edge of $x_1^*$ and $x_2^*$. For sake of simplicity,  we write $e^* = \sigma(e)$. For an edge set $F\subseteq E$, we use the notation $F^*=\{e^* \mid e \in F\}$.

Examples of planar graphings include Schreier graphings of p.m.p.\ actions of finitely generated groups with planar Cayley graphs (in particular surface groups), as well as Palm graphings of planar factor graphs of free invariant point processes on $\R^2$ or $\mathbb{H}^2$~\cite[Remark 17]{mellick2021palm}.

\begin{rem}
    For graphings, we usually assume $\mu(J)=1$ for convenience. For finite graphs this corresponds to normalization by the number of vertices. However, when duals are considered, this normalization is less natural, since the dual graph typically does not have the same number of vertices. It would probably be more natural to normalize by the number of edges which is the same for the two graphs. Nevertheless, to be consistent with the rest of the paper, we do not change the normalization, and simply accept that the vertex measure of the dual graphing is not necessarily 1.  
\end{rem}

\begin{rem}
    Limits of sequences of finite planar graphs are hyperfinite by the Lipton-Tarjan Planar Separation Theorem~\cite{lipton1979separator}. Nevertheless, planar graphings need not be hyperfinite. In fact, there is a sharp dichotomy between hyperfinite and non-hyperfinite planar graphings, see~\cite{angel2018hyperbolic} and~\cite{timar2024full}. 
\end{rem}  

We define the \emph{cocyclic} matroid of $\Gb$ by the infinite analogue of the usual dual rank function formula:
\[\rho^*_{\Gb}(F) = \rho_{\Gb}(E\setminus F) + \eta_\Gb  (F) - \rho_{\Gb} (E).\]
Note that $\rho^*_{\Gb}$ is submodular as it is the sum of a submodular and a modular function. Furthermore, $\rho_{\Gb}(E)\le1$, and equality holds if and only if almost all components of $\Gb$ are infinite. Moreover, the rank of the full edge set in the cocycle matroid is $\rho^*_{\Gb}(E)=\eta_\Gb (E)-1$. There is a nice correspondence between the elements of the associated minorizing measures: \[\alpha \in \bmm^+(\rho_{\Gb}) \iff \eta_\Gb -\alpha \in \bmm^+(\rho^*_{\Gb}).\]

%%%%%%%%%%%%%%%%
\subsection{Duality for hyperfinite planar graphings}
%%%%%%%%%%%%%%%%

For finite planar graphs, the cocycle matroid is the cycle matroid of the dual. In the Borel setting, this holds only for hyperfinite planar graphings. In fact, we will exploit exactly this discrepancy between $\rho_{\Gb}$ and $\rho^*_{\Gb^*}$ in Subsection~\ref{subsec:cost_matroid} to prove Theorem~\ref{thm:cost_achieveing_matroid}.

\begin{rem}
    Theorem~\ref{thm:hyperfinite_duality} is actually a reformulation of Euler's formula, stating that for any embedding of a finite planar graph, the equality $n+f=m+2$ holds, where $n$, $m$ and $f$ denote the number of vertices, edges and faces, respectively. For a planar graphing $\Gb$ with infinite components, the special case $\rho^*_{\Gb^*} (E^*) = \rho_{\Gb}(E)$ translates into $\mu(J)+\mu^*(J^*) =\eta_\Gb (E)$, that is, the $2$ in the formula disappears in the limit.
\end{rem}

We present two proofs for Theorem~\ref{thm:hyperfinite_duality}, one relying on Theorem~\ref{thm:exposed_integer}, the other using the Mass Transport Principle.

\begin{proof}[Proof of Theorem~\ref{thm:hyperfinite_duality} using Theorem~\ref{thm:exposed_integer}]
     First we prove $\rho^*_{\Gb^*}=\rho_\Gb$. We noted earlier that $\bmm^+(\rho^*_{\Gb^*}) = \eta_{\Gb^*}  - \bmm^+(\rho_{\Gb^*})$. By Theorem~\ref{thm:exposed_integer}, we know that exposed points of $\bmm^+(\rho_{\Gb})$ and $\bmm^+(\rho_{\Gb^*})$ are hyperfinite essential spanning forests of $\Gb$ and $\Gb^*$, respectively. Note also that they are in a one-to-one correspondence by complementation. Here we use that $\Gb$ and $\Gb^*$ are hyperfinite, to ensure that the complement is again hyperfinite.

     These observations together show that $\bmm^+(\rho_{\Gb})$ and $\bmm^+(\rho^*_{\Gb^*})$ have the same set of exposed points. This implies $\bmm^+(\rho_{\Gb}) = \bmm^+(\rho^*_{\Gb^*})$, as they are the weak* closure of the convex hull of the same set, see Remark~\ref{rem:exposed_are_dense}. Finally, this implies $\rho_{\Gb} = \rho^*_{\Gb^*}$ by the fact that a submodular function $\varphi$ is determined by $\bmm^+(\varphi)$ in the sense that $\varphi(A)=\sup_{\alpha \in \bmm^+(\varphi)} \alpha(A)$, see~\cite[Corollary 6.10]{lovasz2023submodular}. The proof of $\rho^*_{\Gb} = \rho_{\Gb^*}$ is analogous.
\end{proof}

\begin{proof}[Proof of Theorem~\ref{thm:hyperfinite_duality} using the Mass Transport Principle]
    By using the definitions of $\rho_{\Gb}$ and $\rho^*_{\Gb^*}$ and the equalities $\eta_{\Gb^*}=\eta_\Gb $ and $\rho_{\Gb^*}(E^*)= \mu^*(J^*)$, we obtain that $\rho_{\Gb}(F) = \rho^*_{\Gb^*}(F)$ is equivalent to 
    \begin{equation}\label{eqn:mass_transpont}
        \eta_\Gb(F) = \int_{x \in J} 1 - \frac{1}{|V(\Gb^F_x)|} \diff \mu(x) + \int_{x^* \in J^*} \frac{1}{|V(\Gb^{*,E^*\setminus F^*}_{x^*})|}  \diff \mu^*(x^*).
    \end{equation} 
    To verify the equality, we set up a payment between vertices, faces, and edges. Motivated by the second integral, every face $x^*$ that is in a finite $(E^* \setminus F^*)$-component pays $1/|V(\Gb^{*, E^* \setminus F^*}_{x^*})|$ token to the component. This way every finite $(E^* \setminus F^*)$-component collects 1 token in total, and this token is passed on to an edge (chosen in a Borel way) in the $F$-cycle surrounding the component. We can also make sure that the edges that get payed, oriented from the component that pays to them, form an acyclic oriented graph on the components of $E^* \setminus F^*$. This way the edges of $F$ that do not receive a payment form a subforest $F' \subseteq F$ that has the same connected components in $\Gb$ as $F$. 

    Next, every $x \in J$ in a finite $F$-component pays $1/|V(\Gb^F_x)|$ token to every $F'$-edge in its component. This way all such vertices pay $1-1/|V(\Gb^F_x)|$, and such $F$-edges receive $1$ token in total.

    What remains are infinite $F'$ components and their vertices. Write $F''$ for this edge set, and $J''$ for the vertices. Then $\eta_\Gb (F'') = \mu(J'')$ by Theorem~\ref{lem:levitt}, as $F''$ is a hyperfinite treeing on $J''$ with infinite components. In a mass transport formulation, this means that we can ensure that every vertex in $J''$ pays 1 token, and every edge in $F''$ gets 1 token.
    
    Concluding the above, vertices and faces payed in total the right hand side of \eqref{eqn:mass_transpont}, while every edge in $F$ received 1, so \eqref{eqn:mass_transpont} holds by the Mass Transport Principle.
\end{proof}

%%%%%%%%%%%%%%%%
\subsection{A matroid attaining the cost}
\label{subsec:cost_matroid}
%%%%%%%%%%%%%%%%

\begin{proof}[Proof of Theorem~\ref{thm:cost_achieveing_matroid}]
As we are in the non-hyperfinite case, the dual $\Gb^*$ has no $2$-ended components $\mu^*$-almost surely. By~\cite{conley2021one}, one can pick a 1-ended hyperfinite spanning forest $F^* \subseteq E^*$. The corresponding measure $\alpha_{F^*}$ is in $\bmm^+(\rho_{\Gb^*})$, implying $\eta_{\Gb^*} -\alpha_{F^*} \in \bmm^+(\rho^*_{\Gb^*})$. This measure, considered on $E(\Gb)$, corresponds to $E \setminus F$. The key idea, also exploited in~\cite{conley2021one}, is that $F^*$ being $1$-ended almost everywhere means that $E \setminus F$ is connected on almost every component of $\Gb$. That is, $E \setminus F$ is a treeing of $\Gb$. Therefore, $\rho^*_{\Gb^*} (E) = \alpha_{E\setminus F} (E) = \eta_\Gb  (E \setminus F) = \cost(\Gb)$, where the last inequality holds by Theorem~\ref{thm:gaboriau}.
\end{proof}

%%%%%%%%%%%%%%%%
\subsection{Free versus wired spanning forests}
%%%%%%%%%%%%%%%%

Let $\Gb$ be a graphing with a Borel order $<$ on the edges. We have already seen in Remark~\ref{rem:WmaxSF} that the algorithm for finding bases with respect to $\rho_{\Gb}$ produces the wired maximal spanning forest, denoted by $\Tb_w$. Recall that the free maximal spanning forest is a natural counterpart. We again refer to~\cite[Chapter 11]{lyons2017probability} for a thorough overview, with the caveat that here we treat arbitrary Borel orders on the edge set.

\begin{proof}[Proof of Theorem~\ref{thm:free_vs_wired}]
When $\Gb$ and $\Gb^*$ are dual planar graphings and the Borel total ordering is the same on $E$ and $E^*$, the free maximal spanning forest of $\Gb$ is exactly the complement of the wired minimal spanning forest of $\Gb^*$.
\end{proof}

\begin{rem}
    Based on Theorem~\ref{thm:free_vs_wired}, one could hope to use the free maximal spanning forest for arbitrary, not necessarily planar graphings to define an interesting matroid that goes beyond hyperfinite spanning forests. The problem is that the free maximal spanning forest built on two different Borel orders might have different edge measure, and so they cannot form bases of the same matroid. For example, let $\Gb$ be the Schreier graphing of a free action of $F_2 \times \Z$. Let $<_1$ be such that all $F_2$-edges are larger than all $\Z$-edges, and the ordering is reversed for $<_2$. Then, $\Tb^{<_1}_f$ contains exactly the $F_2$-edges while $\Tb^{<_2}_f$ contains exactly the $\Z$-edges, so their edge measures are $2$ and $1$, respectively. 
\end{rem}

\medskip

\paragraph{Acknowledgement.} The authors are grateful to Miklós Abért, Boglárka Gehér, András Imolay, Gábor Kun, Balázs Maga, Nicolas Monod, Gábor Pete, Carsten Thomassen, Tamás Titkos and Dániel Virosztek for helpful discussions. 

Márton Borbényi was supported by the \'{U}NKP-23-3 New National Excellence Program of the Ministry for Culture and Innovation from the source of the National Research, Development and Innovation Fund (NKFIH). László Márton Tóth was supported by the NKFIH grant KKP-139502, ''Groups and graph limits''. The research was supported by the Lend\"ulet Programme of the Hungarian Academy of Sciences -- grant number LP2021-1/2021, and by Dynasnet European Research Council Synergy project -- grant number ERC-2018-SYG 810115.

%%%%%%%%%%%%%%%%%%%%%%%%%%%%%%%%
\bibliographystyle{abbrv}
\bibliography{graphic}

\end{document}